\documentclass[journal,twoside,web]{ieeecolor}
\usepackage{generic}
\usepackage{cite}
\usepackage{amsmath,amssymb,amsfonts}
\usepackage{algorithmic}
\usepackage{graphicx}
\usepackage{textcomp}
\usepackage{hyperref}
\hypersetup{hidelinks}
\usepackage{bm}
\usepackage{tablefootnote}

\def\BibTeX{{\rm B\kern-.05em{\sc i\kern-.025em b}\kern-.08em
    T\kern-.1667em\lower.7ex\hbox{E}\kern-.125emX}}
\graphicspath{{./figures/}}
\DeclareGraphicsExtensions{.pdf,.eps}

\newtheorem{theorem}{Theorem}[section]
\newtheorem{lemma}[theorem]{Lemma}

\newtheorem{proposition}[theorem]{Proposition}
\newtheorem{assumption}[theorem]{Assumption}

\newtheorem{remark}[theorem]{Remark}

\newcommand{\bz}{\mathbf{z}}
\newcommand{\bx}{\mathbf{x}}
\newcommand{\by}{\mathbf{y}}
\newcommand{\be}{\mathbf{e}}
\newcommand{\blambda}{\bm{\lambda}}
\newcommand{\bdelta}{\bm{\delta}}
\newcommand{\bmu}{\bm{\mu}}

\newcommand{\bE}{\mathbf{E}}

\newcommand{\bS}{\mathbf{S}}

\newcommand{\gradv}[1]{\nabla_{\!#1}}

\begin{document}

    \title{Prescribed-Time Distributed Generalized Nash Equilibrium Seeking}
    \author{Liraz Mudrik, Isaac Kaminer, Sean Kragelund, and Abram H. Clark
    \thanks{This work was supported in part by the Office of Naval Research Science of Autonomy Program under Grant No.\ N0001425GI01545 and Consortium for Robotics Unmanned Systems Education and Research at the Naval Postgraduate School.}
    \thanks{L. Mudrik, I. Kaminer, and S. Kragelund are with the Department of Mechanical and Aerospace Engineering, Naval Postgraduate School, Monterey, CA, 93943, USA (e-mail: liraz.mudrik.ctr@nps.edu; kaminer@nps.edu; spkragel@nps.edu). }
    \thanks{A. H. Clark is with the Department of Physics, Naval Postgraduate School, Monterey, CA, 93943, USA (e-mail: abe.clark@nps.edu).}}
    
    \maketitle
    
    \begin{abstract}
    Safety-critical multi-agent systems, from cooperative guidance to collision avoidance, must often reach a coordinated decision by a hard deadline rather than merely converge to one eventually. This paper proposes the first fully distributed algorithm that solves the generalized Nash equilibrium (GNE) problem, a non-cooperative game with shared coupling constraints and general cost coupling, at a user-prescribed time $T$ independent of initial conditions. The foundation is a centralized, prescribed-time result built on the optimization Lyapunov function framework and implemented via unnormalized Hessian-gradient feedback, chosen because, unlike the Newton and normalized Hessian-gradient realizations, it naturally splits into per-agent computations. Distributing this feedback requires each agent to run three coupled processes simultaneously: a prescribed-time observer of the global state, a local optimization law, and a dual-consensus mechanism that enforces the shared multipliers of the variational GNE. Their simultaneous operation is the core difficulty, as the optimization continually displaces the states the observers track, while estimation errors corrupt the gradients that drive the optimization. We resolve this coupling with a multi-rate gain schedule whose observer and dual-consensus layers contract strictly faster than the optimization layer, so that every error component vanishes exactly at $T$. A Fischer-Burmeister reformulation keeps the design projection-free while enforcing the constraints at the deadline. Numerical results for a Cournot game and a time-critical sensor-coverage problem validate the approach and demonstrate its use as a solver-in-the-loop for time-critical autonomy.
    \end{abstract}
    

\section{Introduction}
\label{sec:introduction}

\IEEEPARstart{O}{ptimization} and equilibrium seeking in multi-agent networks have become fundamental problems in modern control theory, driven by the need to coordinate large-scale decentralized systems. The generalized Nash equilibrium problem (GNEP), where agents optimize individual objective functions subject to shared coupling constraints, is particularly relevant for resource allocation tasks where one agent's decision restricts the feasible set of others~\cite{facchinei_generalized_2010, rosen_existence_1965}. Applications of GNEPs span a wide range of critical infrastructure, including demand-side management in smart grids~\cite{chen_distributed_2021}, coordinated trajectory planning for robotic swarms~\cite{kang_distributed_2026}, and competitive power control in communication networks~\cite{scutari_potential_2006}.

Given the scale, privacy, and hardware requirements of these systems, centralized solutions are often impractical. Consequently, the research focus has shifted to distributed algorithms where agents communicate only with local neighbors. Extensive work has been done on distributed GNE seeking using primal-dual dynamics~\cite{pavel_distributed_2020, yi_distributed_2019}, operator splitting methods~\cite{yi_operator_2019, belgioioso_distributed_2019}, penalty-based approaches~\cite{romano_inexact-penalty_2023, sun_continuous-time_2021}, and consensus-based schemes~\cite{ye_distributed_2017}; see~\cite{hu_distributed_2022,ye_distributed_2023_survey} for comprehensive surveys. These methods handle general cost coupling and shared constraints, but rely on asymptotic convergence.

While asymptotic stability is sufficient for many scenarios, it is often inadequate for safety-critical cyber-physical systems operating under hard real-time constraints. In applications such as missile guidance~\cite{yuan_cooperative_2026}, competitive resource allocation in multi-cluster networks~\cite{meng_linear_2023}, or collision avoidance~\cite{shakouri_prescribed-time_2022}, agents must agree on a valid solution by a strict deadline to prevent system failure. Standard distributed algorithms, which typically exhibit linear or exponential convergence rates~\cite{qu_exponential_2019}, cannot provide such guarantees, as their settling times depend heavily on initial conditions and network connectivity.

To address the need for faster convergence, recent research has explored finite-time~\cite{bhat_finite-time_2000} and fixed-time~\cite{polyakov_nonlinear_2012} stability. Finite-time methods~\cite{cortes_finite-time_2006} guarantee convergence in finite time, but the settling time remains dependent on initial states. Fixed-time methods~\cite{garg_fixed-time_2021} provide a uniform upper bound independent of initial conditions, and have been applied to Nash equilibrium seeking in~\cite{poveda_fixed-time_2023}. However, this settling-time bound is set only indirectly through the gains and is typically a conservative over-estimate of the true convergence time, so it cannot be placed at a specific deadline, limiting its utility in time-critical coordination.

The most recent advancement is prescribed-time control~\cite{song_time-varying_2017, song_prescribed-time_2023_survey}, which employs time-varying gains to force exact convergence at a user-defined time, independent of initial conditions. This paradigm has been successfully applied to consensus problems~\cite{wang_prescribed-time_2019}, centralized optimization~\cite{mudrik_optimization_2025-1}, and distributed Nash equilibrium seeking for standard (uncoupled) games~\cite{zhao_prescribed-time_2023, qian_prescribed-time_2024}. However, none of these methods simultaneously addresses shared coupling constraints and user-prescribed convergence deadlines for GNEPs.

This paper closes that gap. We propose a fully distributed algorithm for general GNE seeking that drives all agents to the variational GNE (v-GNE) exactly at a prescribed time $T$, for the standard GNEP class studied in the distributed GNE literature~\cite{pavel_distributed_2020, yi_distributed_2019, sun_continuous-time_2021}, with arbitrary functional coupling in the costs and shared constraints. This setting is strictly broader than the aggregative games to which the recent predefined-time method of~\cite{guo_distributed_2025} applies, where each cost depends only on local decisions and a scalar aggregate; problems such as pairwise collision avoidance and bilateral contracts lie outside that subclass and force each agent to estimate the full decision profile. The architecture is simultaneous: every agent concurrently runs a prescribed-time observer of the global state, a local optimization law, and a dual-consensus mechanism enforcing the shared multipliers of the v-GNE.

The contributions are fivefold. First, we establish a centralized prescribed-time convergence result for GNE seeking, instantiating the optimization Lyapunov function (OLF) framework of~\cite{mudrik_optimization_2025-1} in an unnormalized Hessian-gradient (uHGD) realization; unlike the Newton and normalized realizations, its right-hand side decomposes into per-agent computations and is therefore implementable distributively. Second, we design a prescribed-time distributed state observer that fixes no leader a priori: each agent pins its own state and is tracked by the rest of the network. Combining the prescribed-time scaling of~\cite{wang_prescribed-time_2019} with the self-anchored structure of~\cite{zou_continuous-time_2021}, it converges exactly even though the anchored states are driven by the optimization, a regime neither reference treats. Third, we obtain the first fully distributed prescribed-time GNE-seeking algorithm, with three concurrently coupled layers (primal consensus, optimization, and dual consensus); the dual-consensus layer, absent from prior prescribed-time NE work~\cite{zhao_prescribed-time_2023, qian_prescribed-time_2024}, is what makes the v-GNE attainable. Fourth, we certify convergence with a single composite time-varying Lyapunov function unifying the optimization residual, consensus errors, and dual disagreements, together with the multi-rate gain schedule that renders their combination dissipative. Fifth, the design is projection-free by using the smoothed Fischer-Burmeister (FB) function~\cite{liao-mcpherson_regularized_2019}, reaching a feasible $\epsilon$-approximate v-GNE at the deadline and recovering the exact v-GNE as $\epsilon \to 0$.
Table~\ref{tab:comparison} summarizes the key distinctions between the proposed distributed algorithm and existing distributed Nash equilibrium-seeking methods.

\begin{table*}[t]
\centering
\caption{Comparison with existing distributed NE/GNE seeking methods.}
\label{tab:comparison}
\renewcommand{\arraystretch}{1.15}
\footnotesize
\begin{tabular}{l c c c c c c c c}
\hline
\textbf{Feature} & \cite{ye_distributed_2017} & \cite{pavel_distributed_2020,yi_operator_2019,yi_distributed_2019,romano_inexact-penalty_2023,sun_continuous-time_2021} & \cite{meng_linear_2023} & \cite{poveda_fixed-time_2023} & \cite{zhao_prescribed-time_2023} & \cite{qian_prescribed-time_2024} & \cite{guo_distributed_2025} & \textbf{Ours} \\
\hline
Game class & NE & GNE & NE & NE & NE & NE & GNE & GNE \\
Shared constraints & $\times$ & \checkmark & $\times$ & $\times$ & $\times$ & $\times$ & \checkmark & \checkmark \\
General cost coupling & \checkmark & \checkmark & \checkmark & \checkmark & \checkmark & \checkmark & $\times$ & \checkmark \\
Convergence rate & Asymptotic & Asymptotic & Exponential & Fixed-time & Prescribed & Prescribed & Predefined & Prescribed \\
Full decision info.\ & \checkmark & \checkmark & \checkmark & \checkmark & \checkmark & \checkmark & $\times$ & \checkmark \\
Graph type & Undirected & Undirected & Undirected & Undirected & Undirected & Directed & Undirected & Undirected \\
\hline
\end{tabular}
\end{table*}

The remainder of this paper is organized as follows. Section~\ref{sec:formulation} formulates the GNE problem and states the standing assumptions. Section~\ref{sec:centralized} develops the centralized GNE seeking result. Section~\ref{sec:algorithm} details the distributed control architecture. Section~\ref{sec:analysis} provides the theoretical convergence analysis. 
Section~\ref{sec:simulations} presents numerical results, followed by concluding remarks.

\section{Problem Formulation}
\label{sec:formulation}
We consider a network of $N$ agents interacting over a communication graph $\mathcal{G} = (\mathcal{V}, \mathcal{E})$, where $\mathcal{V} = \{1, \dots, N\}$ is the set of agents and $\mathcal{E} \subseteq \mathcal{V} \times \mathcal{V}$ is the set of edges representing communication links. The graph is assumed to be undirected and connected. The adjacency matrix $\mathcal{A} = [a_{ij}] \in \mathbb{R}^{N \times N}$ is defined such that $a_{ij} > 0$ if $(j, i) \in \mathcal{E}$ and $a_{ij} = 0$ otherwise. The Laplacian matrix $\mathcal{L} \in \mathbb{R}^{N \times N}$ is defined as $\mathcal{L} = \mathcal{D} - \mathcal{A}$, where $\mathcal{D}$ is the degree matrix. For a connected graph, $\mathcal{L}$ is positive semi-definite, with a single zero eigenvalue corresponding to the eigenvector of ones, $\mathbf{1}_N$; we denote by $\lambda_2(\mathcal{L}) > 0$ the smallest positive eigenvalue (algebraic connectivity). We denote the Kronecker product by $\otimes$ and the Euclidean norm by $\|\cdot\|$. For a stack of vectors $\mathbf{x}_1, \dots, \mathbf{x}_N$, we denote the aggregate vector as $\mathbf{x} = \text{col}(\mathbf{x}_1, \dots, \mathbf{x}_N)$.

The agents engage in a non-cooperative game where each agent $i \in \mathcal{V}$ controls a local decision variable $\mathbf{x}_i \in \mathbb{R}^{n_i}$. The global decision vector is denoted by $\mathbf{x} = \text{col}(\mathbf{x}_1, \dots, \mathbf{x}_N) \in \mathbb{R}^n$, where $n = \sum_{i=1}^N n_i$. Each agent aims to minimize its local objective function $J_i(\mathbf{x}_i, \mathbf{x}_{-i})$, which depends on its own decision $\mathbf{x}_i$ and the decisions of other agents $\mathbf{x}_{-i}$. The optimization is subject to shared coupling constraints. Specifically, the problem for agent $i$ is formulated as:
\begin{subequations}\label{eq:cost_function} 
\begin{align}
    \min_{\mathbf{x}_i} \quad & J_i(\mathbf{x}_i, \mathbf{x}_{-i}) \\
    \text{s.t.} \quad & A\mathbf{x} = \mathbf{b},  \\
    & \mathbf{g}(\mathbf{x}) \leq \mathbf{0}, 
\end{align}
\end{subequations}
where $A \in \mathbb{R}^{m \times n}$ and $\mathbf{b} \in \mathbb{R}^m$ define shared affine equality constraints, and $\mathbf{g}: \mathbb{R}^n \to \mathbb{R}^p$ represents shared inequality constraints.

A strategy profile $\mathbf{x}^* = \text{col}(\mathbf{x}_1^*, \dots, \mathbf{x}_N^*)$ is a GNE if no agent can unilaterally reduce its cost while remaining feasible, given the decisions of all other agents. Formally, for every $i \in \mathcal{V}$,
\begin{equation}
    J_i(\mathbf{x}_i^*, \mathbf{x}_{-i}^*) \leq J_i(\mathbf{x}_i, \mathbf{x}_{-i}^*), \quad \forall\, \mathbf{x}_i \in \mathcal{X}_i(\mathbf{x}_{-i}^*),
    \label{eq:GNE_def}
\end{equation}
where the feasible set available to agent $i$, given the other agents' decisions, is
\begin{equation}
    \mathcal{X}_i(\mathbf{x}_{-i}) = \bigl\{\mathbf{x}_i \in \mathbb{R}^{n_i} : A\mathbf{x} = \mathbf{b},\; \mathbf{g}(\mathbf{x}) \leq \mathbf{0}\bigr\}\Big|_{\mathbf{x} = (\mathbf{x}_i, \mathbf{x}_{-i})},
    \label{eq:feasible_set}
\end{equation}
and $(\mathbf{x}_i, \mathbf{x}_{-i})$ denotes the global decision vector with the $i$-th block replaced by $\mathbf{x}_i$. The coupling in the constraints, where agent $i$'s feasible set $\mathcal{X}_i(\mathbf{x}_{-i})$ depends on the decisions of others, is what distinguishes the GNEP from a standard Nash equilibrium problem.

To guarantee the existence, uniqueness, and solvability of the GNEP, we employ the following standard assumptions regarding the network topology and the optimization structure.

\begin{assumption}[Graph Connectivity] \label{ass:connectivity}
    The communication graph $\mathcal{G}$ is undirected and connected. This ensures that information can propagate between any pair of agents, allowing for the distributed estimation of global variables.
\end{assumption}

\begin{assumption}[Strong Monotonicity] \label{ass:monotonicity}
    Define the pseudo-gradient mapping
    \begin{equation}\label{eq:ps_grad}
        \mathbf{F}(\mathbf{x}) := \text{col}(\gradv{\mathbf{x}_1} J_1, \dots, \gradv{\mathbf{x}_N} J_N).
    \end{equation}
    $\mathbf{F}(\mathbf{x})$ is strongly monotone with respect to $\mathbf{x}$, that is, there exists a constant $m_F > 0$ such that for all $\mathbf{x}, \mathbf{y} \in \mathbb{R}^n$:
    \begin{equation}
        (\mathbf{F}(\mathbf{x}) - \mathbf{F}(\mathbf{y}))^\top (\mathbf{x} - \mathbf{y}) \geq m_F \|\mathbf{x} - \mathbf{y}\|^2.
    \end{equation}
\end{assumption}

\begin{assumption}[Smoothness and Constraint Convexity] \label{ass:smoothness}
    The cost functions $J_i$ and constraint functions $\mathbf{g}$ are twice continuously differentiable, the equality constraint matrix $A$ has full row rank, and each component $g_j$ of $\mathbf{g}$ is convex.
\end{assumption}

\begin{assumption}[Generalized Slater's Condition] \label{ass:slater}
    There exists a strictly feasible point $\bar{\mathbf{x}} \in \mathbb{R}^n$ satisfying:
    \begin{equation}
        A\bar{\mathbf{x}} = \mathbf{b}, \quad \mathbf{g}(\bar{\mathbf{x}}) < \mathbf{0},
    \end{equation}
    where the inequality holds component-wise. This ensures that the Karush-Kuhn-Tucker (KKT) conditions are necessary and sufficient for the v-GNE, and that strong duality holds~\cite{facchinei_generalized_2010}. We also assume that the shared feasible set $\{\bx : A\bx = \mathbf{b},\ \mathbf{g}(\bx) \leq \mathbf{0}\}$ is bounded.
\end{assumption}

The strong monotonicity assumption ensures the existence and uniqueness of the v-GNE~\cite{facchinei_generalized_2010}, and implies global convergence. Under these assumptions, we focus on finding a v-GNE, a refinement of the GNE concept in which all agents share identical dual multipliers for the coupling constraints. A point $\mathbf{x}^*$ is a v-GNE if it solves the variational inequality $\text{VI}(\mathbf{F}, \mathcal{X})$, that is, $\langle \mathbf{F}(\mathbf{x}^*), \mathbf{x} - \mathbf{x}^* \rangle \geq 0$ for all $\mathbf{x} \in \mathcal{X}$, where $\mathcal{X} = \{\mathbf{x} : A\mathbf{x} = \mathbf{b}, \mathbf{g}(\mathbf{x}) \leq \mathbf{0}\}$ is the shared feasible set and $\mathbf{F}$ is the pseudo-gradient mapping. Under Slater's condition, the v-GNE is characterized by the KKT conditions with shared multipliers~\cite{facchinei_generalized_2010}. Let $\blambda \in \mathbb{R}^p$ and $\bmu \in \mathbb{R}^m$ denote the dual multipliers associated with the inequality and equality constraints, respectively. A point $\mathbf{z}^* = \text{col}(\mathbf{x}^*, \blambda^*, \bmu^*)$ is a primal-dual solution if it satisfies:
\begin{subequations}\label{eq:KKT_conditions}
\begin{align}
    &\mathbf{F}(\mathbf{x}^*) + \nabla \mathbf{g}(\mathbf{x}^*)^\top \blambda^* + A^\top \bmu^* = \mathbf{0},  \\
    &A\mathbf{x}^* - \mathbf{b} = \mathbf{0},  \\
    &\mathbf{g}(\mathbf{x}^*)\le \mathbf{0},\ \boldsymbol{\lambda}^*\ge \mathbf{0},\ \langle\boldsymbol{\lambda}^*, \mathbf{g}(\mathbf{x}^*)\rangle=0. \label{eq:KKT_feas_comp}
\end{align}
\end{subequations}
This distinguishes the v-GNE from the broader class of GNEs, in which each agent may hold different dual multipliers that satisfy only local optimality conditions.


\section{Centralized GNE Seeking}
\label{sec:centralized}

This section develops a centralized prescribed-time algorithm for the GNEP, where a single entity has access to the full primal-dual state. The result serves two purposes: it provides the theoretical template that the distributed architecture of Sec.~\ref{sec:algorithm} realizes, and it instantiates the OLF framework of~\cite{mudrik_optimization_2025-1} in the uHGD realization required for distributed implementation.

\subsection{Stationarity Mapping and the Optimization Lyapunov Function}
\label{subsec:olf}

We adopt the control-centric optimization framework introduced in~\cite{mudrik_optimization_2025-1}. Rather than proposing update dynamics and analyzing their convergence post hoc, this approach treats the augmented primal-dual state as a plant with single integrator dynamics
\begin{equation}
    \dot{\bz} = \mathbf{u}, \label{eq:plant}
\end{equation}
where $\mathbf{u}$ is a feedback law to be designed. We define the global augmented state $\bz = \text{col}(\bx, \blambda, \bmu) \in \mathbb{R}^{n+p+m}$. 
An integrator structure is, in fact, necessary for the convergence of first-order algorithms~\cite{scherer_convex_2021}.
The KKT conditions from~\eqref{eq:KKT_conditions} are encoded into a single stationarity vector $\mathbf{S}(\bz) \in \mathbb{R}^{n+p+m}$, defined as:
\begin{equation}
    \mathbf{S}(\bz) = \begin{bmatrix}
        \mathbf{S}_1(\bz) \\
        \mathbf{S}_2(\bz) \\
        \mathbf{S}_3(\bz)
    \end{bmatrix} = \begin{bmatrix}
        \mathbf{F}(\bx) + \nabla \mathbf{g}(\bx)^\top \widetilde{\blambda} + A^\top \bmu \\
        A\bx - \mathbf{b} \\
        \Phi_{\epsilon}(\blambda, -\mathbf{g}(\bx))
    \end{bmatrix},
    \label{eq:stationarity_vector}
\end{equation}
where $\mathbf{S}_1$ captures stationarity, $\mathbf{S}_2$ captures the conditions related to the equality constraints, and $\mathbf{S}_3$ captures the conditions related to the inequality constraints.
The term $\Phi_{\epsilon}(\blambda, -\mathbf{g}(\bx))$ represents the smoothed Fischer-Burmeister function applied element-wise~\cite{fischer_special_1992, liao-mcpherson_regularized_2019}:
\begin{equation}
    \Phi_{\epsilon}(a, b) = \sqrt{a^2 + b^2 + \epsilon^2} - (a + b),
    \label{eq:FB_function}
\end{equation}
with smoothing parameter $\epsilon > 0$. The smoothed FB function satisfies $\Phi_\epsilon(a,b) = 0$ if and only if $a \geq 0$, $b \geq 0$, and $ab = \epsilon^2/2$. 
Consequently, $\mathbf{S}(\bz) = 0$ encodes the $\epsilon$-approximate KKT conditions: stationarity (first block), equality conditions (second block), and relaxed inequality conditions $\blambda_j g_j(\bx) = -\epsilon^2/2$ that recover exactly as $\epsilon \to 0$.

In the stationarity block $\mathbf{S}_1$, the raw multiplier is replaced by the smoothed positive part~\cite{mudrik_optimization_2025-1}
\begin{equation}
    \widetilde{\lambda}_j = \tfrac{1}{2}\!\left(\sqrt{\lambda_j^2 + \epsilon^2} + \lambda_j\right), \qquad j = 1,\ldots,p,
    \label{eq:lambda_tilde}
\end{equation}
a smooth approximation of $\max(\lambda_j,0)$ with $\widetilde{\lambda}_j > 0$ for every $\lambda_j \in \mathbb{R}$; the complementarity block $\mathbf{S}_3$ retains the raw $\blambda$. Because $\widetilde{\lambda}_j > 0$ and each $g_j$ is convex, the curvature term $\sum_j \widetilde{\lambda}_j \nabla^2 g_j(\bx) \succeq 0$ holds unconditionally, so strong monotonicity (Assumption~\ref{ass:monotonicity}) alone makes the generalized Lagrangian Hessian positive definite. The displacement of the resulting equilibrium from the exact KKT point is $O(\epsilon)$~\cite{mudrik_optimization_2025-1}.

We define the OLF for the plant~\eqref{eq:plant} as the squared norm of the stationarity vector:
\begin{equation}
    V(\bz) = \frac{1}{2} \|\mathbf{S}(\bz)\|^2.
    \label{eq:OLF}
\end{equation}
The optimization design objective is to choose $\mathbf{u}$ such that the OLF decays to zero according to a user-prescribed profile. Because constraint satisfaction is enforced by driving $\mathbf{S}$ to zero rather than by projecting onto the feasible set, the resulting dynamics are entirely projection-free. During the transient phase $t < T$, constraint violation is bounded on the sublevel set $\Omega_{V(0)} = \{\bz : V(\bz) \leq V(\bz(0))\}$; at $t = T$, $\mathbf{S}(\bz) = 0$ enforces $A\bx = \mathbf{b}$ exactly and the inequality constraints to within $O(\epsilon)$.

\subsection{Sublevel Set Compactness}
\label{subsec:compactness}

A key requirement for the convergence analysis is that trajectories remain bounded. In standard Lyapunov-based stability analyses, this is ensured by assuming radial unboundedness of the Lyapunov function, so that sublevel sets are compact. The OLF $V = \tfrac{1}{2}\|\mathbf{S}\|^2$ is not radially unbounded for GNEPs with nonlinear inequality constraints, because a multiplier $\lambda_j$ can diverge while the primal variable approaches a point at which the constraint gradient $\nabla g_j$ vanishes (an unconstrained minimizer of the convex function $g_j$), so that the term $\widetilde{\lambda}_j \nabla g_j(\bx)$ in the stationarity residual stays bounded and $V$ remains finite. Compactness is instead secured below a finite threshold $c^*$, as recorded in the following result.

\begin{proposition}[Sublevel Set Compactness]
\label{prop:sufficient_radial}
    Under Assumptions~\ref{ass:monotonicity}--\ref{ass:slater}, there exists a threshold $c^* > 0$ such that for all $0 < c < c^*$, the sublevel set $$\Omega_c = \{\bz \in \mathbb{R}^{n+p+m} : V(\bz) \leq c\}$$ is compact.
    Moreover, if $\mathbf{g}(\mathbf{x}) = \mathbf{C}\mathbf{x} - \mathbf{d}$ where $\mathbf{C} \in \mathbb{R}^{p \times n}$ has full row rank, then $c^* = +\infty$ and $\Omega_c$ is compact for all $c > 0$.
\end{proposition}
The stationarity vector~\eqref{eq:stationarity_vector} with the smoothed multiplier~\eqref{eq:lambda_tilde} is the strongly monotone GNE stationarity vector of~\cite{mudrik_optimization_2025-1}, and Assumptions~\ref{ass:monotonicity}--\ref{ass:slater} supply the hypotheses (strong monotonicity, full row rank of $A$, convexity of $\mathbf{g}$, a bounded shared feasible set, and a Slater point) under which compactness, the existence of $c^*$, and the affine case $c^* = +\infty$ are established in~\cite{mudrik_optimization_2025-1}.
The threshold $c^*$ depends only on the inequality constraints and the feasibility margin, and is independent of the cost functions, communication graph, or algorithm gains. For affine constraints the constraint gradient never vanishes, so $c^* = +\infty$.

\begin{lemma}[Non-Singularity of $\nabla\mathbf{S}$]
\label{lem:nondeg}
Under Assumptions~\ref{ass:monotonicity} and~\ref{ass:smoothness}, for any $\epsilon > 0$, the Jacobian $\nabla\mathbf{S}(\bz)$ is non-singular for every $\bz \in \Omega_c$. Consequently, the uniform lower bound
\begin{equation}
    \underline{\sigma} \triangleq \min_{\bz \in \Omega_c} \sigma_{\min}(\nabla\mathbf{S}(\bz)) > 0
    \label{eq:sigma_lb}
\end{equation}
holds, where compactness of $\Omega_c$ is guaranteed by Proposition~\ref{prop:sufficient_radial}.
\end{lemma}
The proof is deferred to Appendix~\ref{app:nondeg}.

\subsection{Centralized Prescribed-Time Convergence}
\label{subsec:centralized_convergence}

In the centralized setting of~\cite{mudrik_optimization_2025-1}, the plant is the integrator $\dot{\bz} = \mathbf{u}$ of~\eqref{eq:plant}, and three feedback realizations $\mathbf{u}$ are developed: the Hessian-gradient dynamics (HGD), Newton dynamics (ND), and gradient dynamics (GD). The HGD set
\begin{equation}\label{eq:HGD}
    \mathbf{u} = -\sigma(V,t)\,\frac{\nabla V}{\|\nabla V\|^2},
\end{equation}
enforcing $\dot{V} = -\sigma(V,t)$ with equality, and the ND invert the full Jacobian $\nabla \mathbf{S}(\bz)^{-1}$. Both converge to the unique zero of $\mathbf{S}$, but neither feedback is block-separable, even when each agent has access to the full state: the HGD normalization by the global scalar $\|\nabla V\|^2$ couples all blocks nonlinearly, so the per-agent updates do not recombine into the centralized law, and the ND Jacobian inverse mixes all blocks structurally. Neither therefore decomposes into the per-agent computations required for distributed deployment.

We instead use the uHGD feedback, obtained by omitting the $\|\nabla V\|^{-2}$ normalization in~\eqref{eq:HGD}, which sets the feedback law
\begin{equation}
    \mathbf{u} = -\sigma_{opt}(t)\, \nabla V(\bz),
    \label{eq:centralized_GD}
\end{equation}
with
\begin{equation}
    \sigma_{opt}(t) = \frac{\mu_c}{T - t},
    \label{eq:sigma_opt}
\end{equation}
and $\mu_c > 0$. 
This structure can be decomposed, as the $i$-th block of this feedback law, 
\begin{equation}
    \mathbf{u}_i = -\sigma_{opt}(t)\, \nabla_{\bz_i} V,
\end{equation}
depends only on local partial derivatives of $V$ that each agent can evaluate from its state estimates. This makes the uHGD feedback the natural choice for distributed deployment, since the normalized HGD and the ND do not decompose in this way. The price is that the Lyapunov decay is now an inequality $\dot{V} \leq -\sigma(V,t)$ rather than an equality; the resulting convergence rate is quantified in Proposition~\ref{prop:centralized_PT}. Removing the normalization means the gradient-flow dissipation $\dot{V} = -\sigma_{opt}(t)\|\nabla V\|^2$ no longer delivers a decay rate on $V$ directly, as it does for the HGD; the Polyak-{\L}ojasiewicz (PL) inequality $\|\nabla V\|^2 \geq 2\underline{\sigma}^2 V$, which follows from the singular-value bound of Lemma~\ref{lem:nondeg}, is precisely what converts this dissipation into the prescribed-time rate of Proposition~\ref{prop:centralized_PT}, and the same mechanism underlies the distributed analysis.

\begin{proposition}[Centralized Prescribed-Time GNEPs]
\label{prop:centralized_PT}
Under Assumptions~\ref{ass:monotonicity}--\ref{ass:slater}, consider the uHGD feedback~\eqref{eq:centralized_GD} applied to the plant~\eqref{eq:plant}. If $V(\bz(0)) < c^*$, then:
\begin{enumerate}
    \item[(i)] All trajectories remain bounded: $\bz(t) \in \Omega_{V(0)}$ for all $t \in [0, T)$.
    \item[(ii)] The OLF satisfies
    \begin{equation}
        V(\bz(t)) \leq V(\bz(0)) \left(\frac{T-t}{T}\right)^\gamma, \quad \gamma = 2 \underline{\sigma}^2 \mu_c,
        \label{eq:centralized_decay}
    \end{equation}
    where $\underline{\sigma} = \min_{\bz \in \Omega_{V(0)}} \sigma_{\min}(\nabla \mathbf{S}(\bz)) > 0$.
    \item[(iii)] Consequently, $\lim_{t \to T^-} V(\bz(t)) = 0$, and at time $t = T$ the state $\bz(t)$ reaches the unique zero of the $\epsilon$-smoothed stationarity vector.
\end{enumerate}
\end{proposition}
The proof is deferred to Appendix~\ref{app:centralized_PT}.
The challenge addressed in the remainder of this paper is to realize the centralized gradient feedback~\eqref{eq:centralized_GD} through a fully distributed architecture, where no single agent has access to the full state $\bz$.

\section{Distributed Control Architecture}
\label{sec:algorithm}

In this section, we design a distributed feedback controller that solves the GNEP~\eqref{eq:cost_function} in prescribed time $T > 0$. Guided by the centralized result of Sec.~\ref{sec:centralized}, we realize the centralized feedback~\eqref{eq:centralized_GD} through a simultaneous architecture where agents optimize their local decision variables while concurrently estimating the global state required for gradient computation.

To handle the coupled constraints and the non-separable cost functions, each agent must estimate the full primal-dual state of the network. We define the local augmented state for agent $i$ as:
\begin{equation}
    \bz_i(t) = \text{col}(\bx_i(t), \blambda_i(t), \bmu_i(t)) \in \mathbb{R}^{m_i},
    \label{eq:local_state}
\end{equation}
where $\bx_i \in \mathbb{R}^{n_i}$ is agent $i$'s decision variable, and $\blambda_i \in \mathbb{R}^p$, $\bmu_i \in \mathbb{R}^m$ are agent $i$'s local copies of the shared dual multipliers for the inequality and equality constraints, respectively. The total local dimension is $m_i = n_i + p + m$.

Additionally, each agent $i$ maintains an estimate of the entire network's augmented state. To distinguish this cyber-layer estimate from the physical local variables, we denote the estimate vector by $\by_i(t)$:
\begin{equation}
    \by_i(t) = \text{col}(\by_{i1}(t), \dots, \by_{iN}(t)) \in \mathbb{R}^{\bar{m}},
    \label{eq:estimated_state}
\end{equation}
where $\by_{ij}(t) \in \mathbb{R}^{m_j}$ represents agent $i$'s estimate of agent $j$'s local augmented state $\bz_j(t)$. The total dimension is $\bar{m} = \sum_{j=1}^N m_j$.

\subsection{Prescribed-Time Distributed State Observer}
\label{subsec:consensus}
The algorithm does not require choosing a leader a priori. Instead, each agent $i$ runs a distributed prescribed-time consensus-based observer that estimates the states $\by_{ij}$ of all agents $j$, while pinning its own estimate to its true state, $\by_{ii} = \bz_i$. Each agent therefore acts as the anchor for its own block and as a follower for every other block. The gain structure follows the prescribed-time time-scaling framework of~\cite{song_time-varying_2017, wang_prescribed-time_2019}, and the pinning follows the asymptotic observer of~\cite{zou_continuous-time_2021}; their combination when the anchored states are driven by the optimization dynamics is a contribution of this work, with the formal guarantee provided by Theorem~\ref{thm:main}. For each target agent $j \in \mathcal{V}$, agent $i$ updates its estimate $\by_{ij}$ according to:
\begin{equation}
    \dot{\by}_{ij}(t) = -\left(k_o + c_o\, \mu_o(t)\right) \sum_{k=1}^N a_{ik} (\by_{ij}(t) - \by_{kj}(t)),
    \label{eq:consensus_dynamics}
\end{equation}
for all $i \neq j$. Here, $k_o, c_o > 0$ are tuning gains and $\mu_o(t)$ is the time-scaling function
\begin{equation}
    \mu_o(t) = \left(\frac{T}{T - t}\right)^{\gamma_c}, \quad \gamma_c > 1,
    \label{eq:consensus_scaling}
\end{equation}
which grows unbounded as $t \to T$. The effective consensus gain in~\eqref{eq:consensus_dynamics} is therefore $\xi(t) = k_o + c_o\, \mu_o(t)$, which diverges at the polynomial order $\gamma_c > 1$. This order exceeds that of the optimization gain~\eqref{eq:opt_scaling}, so the observer behaves as a fast inner loop relative to the optimization trajectory it tracks (Sec.~\ref{sec:analysis}).

The consensus dynamics~\eqref{eq:consensus_dynamics} alone implement a leaderless consensus: for each target agent~$j$, the estimates $\{\by_{ij}\}_{i=1}^N$ would converge to a common value determined by the initial conditions $\{\by_{ij}(0)\}_{i=1}^N$, which generally differs from the true state~$\bz_j(t)$. To anchor the estimates to the correct value, agent~$j$ pins its own estimate to its current state~\cite{zou_continuous-time_2021}:
\begin{equation}
    \by_{jj}(t) \equiv \bz_j(t), \quad \forall\, j \in \mathcal{V},\; \forall\, t \in [0, T).
    \label{eq:pinning}
\end{equation}
Since agent~$j$ has direct access to~$\bz_j(t)$, no estimation is needed for its own state.

The pinning condition~\eqref{eq:pinning} induces $N$ parallel tracking problems: for the $j$-th block of the global state vector, agent $j$ is the anchor while all other agents track its state via~\eqref{eq:consensus_dynamics}. This creates a bidirectional coupling: the optimization dynamics drive the anchored state~$\bz_j(t)$, while the consensus protocol propagates this change to the tracking agents, with the prescribed-time gain ensuring that tracking errors vanish at~$T$. The gain dominance condition $c_o > c_o^*$ in Theorem~\ref{thm:main} is precisely what ensures the consensus layer absorbs this optimization-induced excitation.

\subsection{Optimization Feedback Law with Dual Consensus}
\label{subsec:opt_layer}

Each agent's local augmented state evolves as
\begin{equation}
    \label{eq:z_i}
    \dot{\bz}_i = \mathbf{u}_i,
\end{equation}
where $\mathbf{u}_i = \text{col}(\mathbf{u}_i^x, \mathbf{u}_i^\lambda, \mathbf{u}_i^\mu)$ is the distributed feedback law designed so that all agents reach the v-GNE at the prescribed time~$T$. The three components of~$\mathbf{u}_i$ are defined as follows.

\textbf{Primal Feedback:} Agent $i$ updates its decision variable
$\bx_i$ via gradient flow on the OLF:
\begin{equation}
    \mathbf{u}_i^x(t) = -\sigma_{opt}(t) \gradv{\bx_i} V(\by_i(t)).
    \label{eq:primal_dynamics}
\end{equation}
The gradient is evaluated at the consensus
estimate~$\by_i(t)$ from~\eqref{eq:consensus_dynamics}, not the true
network state; the prescribed-time consensus layer ensures this
substitution becomes exact at~$T$.

\textbf{Dual Updates with Consensus:} Agent $i$ updates its dual
variables $\blambda_i$ and $\bmu_i$ using both gradient descent and
consensus with neighboring agents:
\begin{align}
    \mathbf{u}_i^\lambda(t) = -N\,&\sigma_{opt}(t) \gradv{\blambda_i}
    V(\by_i(t)) \notag\\
    &- \kappa(t) \sum_{k=1}^N a_{ik} (\blambda_i(t) - \blambda_k(t)),
    \label{eq:dual_lambda_dynamics} \\
    \mathbf{u}_i^\mu(t) = -N\,&\sigma_{opt}(t) \gradv{\bmu_i}
    V(\by_i(t)) \notag\\
    &- \kappa(t) \sum_{k=1}^N a_{ik} (\bmu_i(t) - \bmu_k(t)),
    \label{eq:dual_mu_dynamics}
\end{align}
where $\sigma_{opt}(t)$ and $\kappa(t)$ are time-varying gains defined below in~\eqref{eq:opt_scaling}--\eqref{eq:dual_consensus_scaling}.
The consensus terms in \eqref{eq:dual_lambda_dynamics}--\eqref{eq:dual_mu_dynamics} force $\blambda_i \to \blambda_j$ and $\bmu_i \to \bmu_j$ for all $i, j$, ensuring the shared dual multiplier requirement of the v-GNE; without them, each agent's dual variables would converge to different values satisfying only local KKT conditions.

\begin{remark}[Dual scaling factor]
\label{rem:N_factor}
The factor $N$ in the dual gradient terms restores the aggregate dual descent to the centralized rate. The multipliers enter the centralized $\mathbf{S}$ as a single shared $(\blambda, \bmu)$, identified in the distributed setting with the averages $\bar{\blambda} = \frac{1}{N}\sum_i \blambda_i$ and $\bar{\bmu} = \frac{1}{N}\sum_i \bmu_i$. In the stationarity block $\mathbf{S}_1$ they enter only through these averages, so by the chain rule $\partial \mathbf{S}_1/\partial \blambda_i = \frac{1}{N}\,\partial \mathbf{S}_1/\partial \bar{\blambda}$ (and likewise for $\bmu_i$), and the explicit $N$ cancels this $\frac{1}{N}$ exactly. The equality multiplier $\bmu$ enters $V$ only through $\mathbf{S}_1$, so its rescaling is exact; the inequality multiplier also enters the complementarity block $\mathbf{S}_3 = \Phi_\epsilon(\blambda, -\mathbf{g}(\bx))$ through each agent's own copy $\blambda_i$, a contribution that the scalar $N$ does not reconcile but that forms a per-agent residual vanishing as $\mathbf{S}_3 \to \mathbf{0}$ at the deadline. The consensus terms in~\eqref{eq:dual_lambda_dynamics}--\eqref{eq:dual_mu_dynamics} drop out of the average since $\mathbf{1}^\top\mathcal{L} = \mathbf{0}$ and vanish once the copies agree, so the aggregate descent $\frac{d\bar{\blambda}}{dt} = \frac{1}{N}\sum_i \dot{\blambda}_i$ recovers the centralized dual flow of Proposition~\ref{prop:centralized_PT} in the limit.
\end{remark}

The optimization and dual consensus gains contain two tuning parameters, $\mu_c$ and $k_d$, embedded in time-varying gains that grow unbounded as $t \to T$. Together with the consensus layer parameter~$c_o$ from~\eqref{eq:consensus_scaling}, these three parameters govern the prescribed-time convergence.
The remaining gains are the optimization layer gain:
\begin{equation}
    \sigma_{opt}(t) = \frac{\mu_c}{T - t}, \quad \mu_c > 0,
    \label{eq:opt_scaling}
\end{equation}
and the dual consensus gain:
\begin{equation}
    \kappa(t) = k_d\, \mu_o(t) = k_d \left(\frac{T}{T-t}\right)^{\gamma_c}, \quad k_d > 0.
    \label{eq:dual_consensus_scaling}
\end{equation}
The dual consensus gain shares the time-scaling function $\mu_o$ of the observer layer, so it diverges at the same order $\gamma_c$ and drives the per-agent multipliers to agreement strictly faster than the optimization layer evolves; this multi-rate separation is what the ensuing dissipation analysis exploits.

The observer and dual consensus gains both diverge at order $\gamma_c > 1$, strictly faster than the order-one optimization gain, so that estimation and dual-disagreement errors decay faster than the optimization dynamics can excite them. The constants $c_o$ and $k_d$ must additionally exceed finite thresholds to dominate the coupling over the transient.

\subsection{Error Coordinates and Closed-Loop Dynamics}
\label{subsec:error_dynamics}
The distributed controller presented above couples three dynamical subsystems. We introduce error coordinates that decompose the closed-loop behavior into consensus tracking, optimization progress, and dual agreement.

\textbf{Consensus Error.} 
Let
\begin{equation}
    \be_{ij}(t) = \by_{ij}(t) - \bz_j(t)
\end{equation}
denote the error in agent $i$'s estimate of agent $j$'s state. By the pinning condition~\eqref{eq:pinning}, $\be_{jj}(t) \equiv 0$, so the error in the $j$-th block is carried entirely by the $N-1$ tracking agents; we stack their errors as $\tilde{\bE}_j = \text{col}(\be_{ij})_{i \neq j} \in \mathbb{R}^{(N-1) m_j}$. Because agent $j$ is held fixed, the consensus operator acting on $\tilde{\bE}_j$ is the grounded (Dirichlet) Laplacian $\mathcal{L}_g^{(j)} \in \mathbb{R}^{(N-1)\times(N-1)}$, the principal submatrix of $\mathcal{L}$ obtained by deleting its $j$-th row and column. For a connected graph, $\mathcal{L}_g^{(j)} \succ 0$~\cite{pirani2015smallest}; we write $\theta_{\min} = \min_{j \in \mathcal{V}} \lambda_{\min}(\mathcal{L}_g^{(j)}) > 0$.

\textbf{Dual Disagreement.} Let $\bar{\blambda}(t) = \frac{1}{N} \sum_{i=1}^N \blambda_i(t)$ and $\bar{\bmu}(t) = \frac{1}{N} \sum_{i=1}^N \bmu_i(t)$ denote the average dual variables. Define the disagreements $\bdelta_i^\lambda(t) = \blambda_i(t) - \bar{\blambda}(t)$ and $\bdelta_i^\mu(t) = \bmu_i(t) - \bar{\bmu}(t)$. Stack these as $\bdelta = \text{col}(\bdelta^\lambda, \bdelta^\mu) \in \mathbb{R}^{N(p+m)}$.

\textbf{Optimization Residual.} Since the centralized $V$ depends on a single shared $(\blambda, \bmu)$, identified with the network averages in the distributed setting, we define $W_o(t) = V(\bx(t), \bar{\blambda}(t), \bar{\bmu}(t))$; per-agent dual deviations are captured by~$W_\delta$ below.

In these coordinates, the closed-loop dynamics take the form:
\begin{align}
    \dot{\tilde{\bE}}_j &= -\xi(t) (\mathcal{L}_g^{(j)} \otimes I) \tilde{\bE}_j 
        - \mathbf{1}_{N-1} \otimes \mathbf{u}_j, 
        \label{eq:error_consensus}\\
    \dot{\bdelta} &= -\kappa(t) (\mathcal{L} \otimes I) \bdelta 
        + \boldsymbol{\eta}(t), 
        \label{eq:error_dual} \\
    \dot{W}_o &= \nabla_{\bx} V^\top \mathbf{u}^x 
        + \nabla_{\bar{\blambda}} V^\top \bar{\mathbf{u}}^\lambda 
        + \nabla_{\bar{\bmu}} V^\top \bar{\mathbf{u}}^\mu, 
        \label{eq:error_opt}
\end{align}
where $\mathbf{1}_{N-1} \otimes \mathbf{u}_j$ reflects that every tracking agent $i \neq j$ is forced by the anchor velocity $\dot{\bz}_j = \mathbf{u}_j$, $\bar{\mathbf{u}}^\lambda = \frac{1}{N}\sum_i \mathbf{u}_i^\lambda$ and $\bar{\mathbf{u}}^\mu = \frac{1}{N}\sum_i \mathbf{u}_i^\mu$ are the averaged dual controls, and $\boldsymbol{\eta}(t)$ is the gradient perturbation arising from heterogeneous gradient evaluations across agents. The three subsystems are bidirectionally coupled: the optimization dynamics perturb the consensus layer through~$\mathbf{u}_j$ in~\eqref{eq:error_consensus}, while the consensus estimation errors corrupt the gradients in~\eqref{eq:error_opt} and drive dual disagreement through~$\boldsymbol{\eta}(t)$ in~\eqref{eq:error_dual}.

The composite Lyapunov function combines the optimality residual with the network disagreement energy:
\begin{equation}
    W(t) = \underbrace{W_o(t)}_{\text{optimality (slow)}} + \underbrace{W_c(t) + k_d\, W_\delta(t)}_{\text{network disagreement (fast)}}.
    \label{eq:composite_preview}
\end{equation}
The optimality residual $W_o$ descends along the optimization gradient at the order-one rate $\sigma_{opt}$ and sets the prescribed-time rate, while the consensus and dual disagreements $W_c, W_\delta$ are driven to zero by the order-$\gamma_c$ gains $\xi$ and $\kappa = k_d\mu_o$ (see~\eqref{eq:consensus_scaling},~\eqref{eq:dual_consensus_scaling}), strictly faster. The control design objective in error coordinates is to establish the dissipation inequality
\begin{equation}
    \dot{W}(t) \leq -\frac{\gamma}{T - t}\, W(t), \quad \gamma > 0,
    \label{eq:target_dissipation}
\end{equation}
which, upon integration, yields $W(t) \leq W(0)\left(\frac{T-t}{T}\right)^\gamma \to 0$ as $t \to T^-$. 
We establish below that tuning parameters $(c_o, \mu_c, k_d)$ satisfying~\eqref{eq:target_dissipation} exist for any prescribed time~$T$.

\section{Distributed Convergence Analysis}
\label{sec:analysis}

In this section, we establish that the distributed optimizer from Sec.~\ref{sec:algorithm} achieves prescribed-time convergence to the unique v-GNE. The analysis strategy is to construct a composite Lyapunov function $W(t)$ and show, via a contradiction argument, that $W(t) \leq W(0) < c^*$ throughout $[0,T)$. This confinement ensures $W_o \leq W < c^*$, so the centralized state $(\bx, \bar{\blambda}, \bar{\bmu})$ remains in the compact sublevel set of Proposition~\ref{prop:sufficient_radial}, and all constants from Sec.~\ref{sec:centralized} are valid throughout.

\subsection{Main Convergence Result}

\begin{theorem}[Prescribed-Time Convergence to v-GNE]
\label{thm:main}
Consider the GNEP~\eqref{eq:cost_function} under Assumptions \ref{ass:connectivity}--\ref{ass:slater}. Let agents implement the distributed controller~\eqref{eq:consensus_dynamics}--\eqref{eq:dual_consensus_scaling}. If $W(0) < c^*$, then there exist finite thresholds $c_o^*, k_d^* > 0$ such that for all $c_o > c_o^*$ and $k_d > k_d^*$:
\begin{enumerate}
    \item {Consensus:} $\lim_{t \to T^-} \|\by_{ij}(t) - \bz_j(t)\| = 0$ for all $i, j$.
    \item {Optimality:} $\lim_{t \to T^-} V(\bx(t), \bar{\blambda}(t), \bar{\bmu}(t)) = 0$.
    \item {Dual Agreement:} $\lim_{t \to T^-} \|\blambda_i(t) - \blambda_j(t)\| = 0$ and $\lim_{t \to T^-} \|\bmu_i(t) - \bmu_j(t)\| = 0$ for all $i, j$.
    \item {Boundedness:} All signals remain bounded on $[0, T)$.
\end{enumerate}
Consequently, the system converges at time $T$ to the unique point $\bz^* = (\bx^*, \blambda^*, \bmu^*)$ with $\mathbf{S}(\bz^*) = 0$, the zero of the $\epsilon$-smoothed stationarity vector.
\end{theorem}
This $\bz^*$ is an $\epsilon$-approximate v-GNE: by the standard FB smoothing estimate~\cite{liao-mcpherson_regularized_2019}, it lies within $O(\epsilon)$ of the exact solution $\bz^*_{\text{exact}}$ and recovers it as $\epsilon \to 0$, under the regularity conditions of~\cite{liao-mcpherson_regularized_2019}. The gain thresholds $c_o^*$ and $k_d^*$ depend on the problem constants and are met by taking $c_o$ and $k_d$ sufficiently large; for affine inequality constraints $c^* = +\infty$ and the requirement $W(0) < c^*$ holds trivially.

\subsection{Proof of Theorem \ref{thm:main}}
\label{subsec:proof}
The proof constructs the composite Lyapunov function $W(t) = W_o(t) + W_c(t) + k_d W_\delta(t)$ from~\eqref{eq:composite_preview} and proceeds in five steps: (1) definition of the Lyapunov components, (2) a PL inequality for the optimality residual, (3) dissipation of the optimality residual, (4) dissipation of the network disagreement and combination, and (5) trajectory confinement and integration to obtain prescribed-time convergence. Steps~1--4 are carried out under the standing assumption that $W(t) \leq W(0)$; Step~5 verifies that this assumption is self-consistent for all $t \in [0,T)$.

Since $W(0) < c^*$ by hypothesis, the sublevel set $\Omega_{W(0)} = \{\bz : V(\bz) \leq W(0)\}$ is compact by Proposition~\ref{prop:sufficient_radial} (note $\Omega_{W(0)} \supseteq \Omega_{V(0)}$ since $W(0) \geq V(\bz(0))$; for affine constraints, $c^* = +\infty$ and the distinction is vacuous). On this set, Lemma~\ref{lem:nondeg} yields the uniform lower bound $\sigma_{\min}(\nabla\mathbf{S}(\bz)) \geq \underline{\sigma} > 0$, and $\nabla V$ is Lipschitz with constant $L_V > 0$ by Assumption~\ref{ass:smoothness}. With slight abuse of notation, $\underline{\sigma}$ and $L_V$ henceforth refer to constants computed over $\Omega_{W(0)}$.

\subsubsection*{Step 1: Composite Lyapunov Function}

Recall the composite structure from~\eqref{eq:composite_preview}: $W(t) = W_o(t) + W_c(t) + k_d W_\delta(t)$. Using the error coordinates from Section~\ref{subsec:error_dynamics}, the components are:
\begin{align}
    W_c(t) &= \frac{1}{2} \sum_{j=1}^N \tilde{\bE}_j^\top (\mathcal{L}_g^{(j)} \otimes I_{m_j}) \tilde{\bE}_j, \label{eq:Wc} \\
    W_o(t) &= V(\bx(t), \bar{\blambda}(t), \bar{\bmu}(t)), \label{eq:Wo} \\
    W_\delta(t) &= \frac{1}{2} \bdelta^\top (\mathcal{L} \otimes I_{p+m}) \bdelta. \label{eq:Wdelta}
\end{align}
The optimality residual $W_o$ is the centralized OLF evaluated at the network-averaged dual variables, while the consensus and dual energies $W_c, W_\delta$ quantify the primal-estimation and dual-multiplier disagreements among agents.

Since $\mathcal{G}$ is connected, $\mathcal{L}$ has a simple zero eigenvalue with eigenvector $\mathbf{1}_N$, and all other eigenvalues are positive. The dual disagreement $\bdelta$ is orthogonal to $\mathbf{1}_N \otimes I$ by construction (since $\sum_i \bdelta_i = 0$), so the Rayleigh quotient of $\mathcal{L} \otimes I$ on this subspace is bounded below by $\lambda_2(\mathcal{L})$, yielding the standard coercive bound:
\begin{equation}
    W_\delta \geq \frac{\lambda_2(\mathcal{L})}{2} \|\bdelta\|^2. \label{eq:Wdelta_coercive}
\end{equation}
For $W_c$, the grounded Laplacian is positive definite, $\mathcal{L}_g^{(j)} \succeq \theta_{\min} I$, so each summand satisfies $\tilde{\bE}_j^\top(\mathcal{L}_g^{(j)} \otimes I_{m_j})\tilde{\bE}_j \geq \theta_{\min}\|\tilde{\bE}_j\|^2$, yielding the coercive bound:
\begin{equation}
    W_c \geq \frac{\theta_{\min}}{2} \sum_j \|\tilde{\bE}_j\|^2. \label{eq:Wc_coercive}
\end{equation}
Thus, $W(t) = 0$ if and only if consensus, optimality, and dual agreement are simultaneously achieved.

\subsubsection*{Step 2: PL Inequality for the Optimality Residual}

Because the uHGD provides only the inequality $\dot{V} \leq -\sigma(V,t)$, a decay rate on the optimality residual cannot be read off directly; as in the centralized case, a PL inequality~\cite{karimi_linear_2016} is required. Since $V$ depends only on the centralized variables $(\bx, \bar{\blambda}, \bar{\bmu})$, the chain rule together with the $N$-scaling in~\eqref{eq:dual_lambda_dynamics}--\eqref{eq:dual_mu_dynamics} aggregates the per-agent dissipation to $\sum_{i=1}^N [\|\nabla_{\bx_i} V\|^2 + N\|\nabla_{\blambda_i} V\|^2 + N\|\nabla_{\bmu_i} V\|^2] = \|\nabla_{(\bx,\bar{\blambda},\bar{\bmu})} V\|^2$, where $V$ is evaluated at the averaged duals as in Remark~\ref{rem:N_factor}. By Lemma~\ref{lem:nondeg}, $\|\nabla V\|^2 = \|\nabla\mathbf{S}^\top\mathbf{S}\|^2 \geq \underline{\sigma}^2\|\mathbf{S}\|^2 = 2\underline{\sigma}^2 V$, which applied to $V = W_o$ gives the PL inequality
\begin{equation}
    \|\nabla_{(\bx,\bar{\blambda},\bar{\bmu})} V\|^2 \geq 2\underline{\sigma}^2\, W_o.
    \label{eq:PL_Vnet}
\end{equation}
The disagreement directions $\blambda_i - \bar{\blambda}$ lie in the null space of $\nabla V$ and carry no optimality dissipation; they are instead driven to zero by the consensus and dual gains, as quantified in Step~4.

\subsubsection*{Step 3: Dissipation of the Optimality Residual}

We first bound the optimality dissipation. Differentiating $W_o = V(\bx, \bar{\blambda}, \bar{\bmu})$ along the distributed dynamics, the consensus terms in~\eqref{eq:dual_lambda_dynamics}--\eqref{eq:dual_mu_dynamics} vanish in the dual average (since $\frac{1}{N}\sum_i (\mathcal{L}\otimes I)\blambda = \mathbf{0}$), so the gradient-descent contributions aggregate to $-\sigma_{opt}(t)\|\nabla_{(\bx,\bar{\blambda},\bar{\bmu})} V\|^2$. Each agent evaluates its gradient at its network-state estimate $\by_i$ rather than at the true network state $\bz^{\mathrm{net}} = \mathrm{col}(\bz_1,\dots,\bz_N)$ (with a slight abuse of notation, $V(\by_i)$ denotes the centralized OLF evaluated at agent $i$'s reconstructed global primal estimate and its own dual copies $(\blambda_i,\bmu_i)$), introducing a Lipschitz perturbation $\bm{\eta}_i = \nabla V(\by_i) - \nabla V(\bz^{\mathrm{net}})$ with $\|\bm{\eta}_i\| \leq L_V \|\by_i - \bz^{\mathrm{net}}\|$; since $\|\by_i - \bz^{\mathrm{net}}\|^2 = \sum_j \|\be_{ij}\|^2$, the perturbation is controlled by $W_c$ through~\eqref{eq:Wc_coercive}. Since $V = \tfrac{1}{2}\|\mathbf{S}\|^2$ and $\|\nabla\mathbf{S}\| \leq M_S$ on $\Omega_{W(0)}$, $\|\nabla V\| \leq M_S\sqrt{2 W_o}$, so each cross-term $\langle \nabla V, \bm{\eta}_i\rangle$ is bounded by $C_o\sqrt{W_o}\,\sqrt{W_c}$. Applying the PL inequality~\eqref{eq:PL_Vnet},
\begin{equation}
    \dot{W}_o \leq -2\underline{\sigma}^2\, \sigma_{opt}(t)\, W_o + C_o\, \sigma_{opt}(t)\, \sqrt{W_o}\,\sqrt{W_c},
    \label{eq:Vnet_dissipation}
\end{equation}
where $C_o > 0$ depends on $M_S$, $L_V$, $N$, and $\lambda_2(\mathcal{L})$. The cross-term couples the optimality residual to the consensus error and is absorbed in Step~4.

\subsubsection*{Step 4: Dissipation of the Network Disagreement and Combination}

Both disagreement energies are driven by the order-$\gamma_c$ gains, following the prescribed-time observer framework of~\cite{wang_prescribed-time_2019}. Differentiating $W_c$ along~\eqref{eq:error_consensus} and $W_\delta$ along~\eqref{eq:error_dual}, and applying $(\mathcal{L}_g^{(j)})^2 \succeq \theta_{\min}\,\mathcal{L}_g^{(j)}$ to the consensus terms and $(\mathcal{L}\otimes I)^2 \succeq \lambda_2(\mathcal{L})(\mathcal{L}\otimes I)$ on the disagreement subspace to the dual terms, gives
\begin{align}
    &\dot{W}_c \leq -2\xi\theta_{\min} W_c + C_1\sigma_{opt}\sqrt{W_c}\sqrt{W_o} + C_2\kappa\sqrt{W_c}\sqrt{W_\delta}, \label{eq:Wc_dot_new}\\
    &k_d\dot{W}_\delta \leq -2k_d^2\lambda_2\mu_o W_\delta + C_3 k_d\sigma_{opt}\sqrt{W_c}\sqrt{W_\delta}. \label{eq:Wdelta_dot}
\end{align}
In~\eqref{eq:Wc_dot_new} the forcing is the anchor velocity $\dot{\bz}_j = \mathbf{u}_j$, whose primal and dual-descent part is $O(\sigma_{opt})$ (scaling with $\sqrt{W_o}$) and whose dual-consensus part is $O(\kappa)$ (scaling with $\sqrt{W_\delta}$); in~\eqref{eq:Wdelta_dot} the forcing $\bm{\eta}_\delta = O(\sigma_{opt}\sqrt{W_c})$ is the heterogeneous dual descent. Because $\xi$ and $\kappa = k_d\mu_o$ share the order $\gamma_c$ and $\sigma_{opt}/\mu_o \to 0$, Young's inequality absorbs the two cross-terms proportional to $\sqrt{W_c}\sqrt{W_\delta}$ into the order-$\gamma_c$ dissipations for $c_o, k_d$ above finite thresholds, leaving only the optimality coupling proportional to $\sigma_{opt}\sqrt{W_c}\sqrt{W_o}$.

Summing~\eqref{eq:Vnet_dissipation},~\eqref{eq:Wc_dot_new}, and~\eqref{eq:Wdelta_dot}, and absorbing this last cross-term by Young's inequality $C\sigma_{opt}\sqrt{W_c}\sqrt{W_o} \leq \tfrac{1}{2}\xi\theta_{\min} W_c + \tfrac{C^2\sigma_{opt}^2}{2\xi\theta_{\min}}W_o$,
\begin{align}
    \dot{W} \leq &-\Bigl(2\underline{\sigma}^2 - \tfrac{C^2\sigma_{opt}(t)}{2\xi(t)\theta_{\min}}\Bigr)\sigma_{opt}(t)\,W_o - \tfrac{1}{2}\xi(t)\theta_{\min} W_c \notag \\
    & - k_d^2\lambda_2\mu_o(t)\,W_\delta.
    \label{eq:combined_dissipation}
\end{align}
Since $\xi(t)=k_o + c_o(T/(T-t))^{\gamma_c}$ and $\sigma_{opt}(t)=\mu_c/(T-t)$, the ratio $\sigma_{opt}(t)/\xi(t) = \mu_c(T-t)^{\gamma_c-1}/[k_o(T-t)^{\gamma_c}+c_o T^{\gamma_c}]\to 0$ as $t\to T$ for any $\gamma_c > 1$, and is uniformly bounded on $[0,T)$ (the same holds for $\sigma_{opt}/\kappa$). Hence the $W_o$ coefficient in~\eqref{eq:combined_dissipation} is strictly positive for $c_o$ above a finite threshold $c_o^*$; and since $\mu_o/\sigma_{opt}\to\infty$, both $\tfrac12\xi\theta_{\min}$ and $k_d\lambda_2\mu_o$ dominate $2\underline{\sigma}^2\sigma_{opt}$ for $c_o, k_d$ above finite thresholds, so every term decays at no less than the optimality rate. With $\sigma_{opt}=\mu_c/(T-t)$ this yields, for all $c_o>c_o^*$ and $k_d>k_d^*$,
\begin{equation}
    \dot{W}(t) \leq -\frac{\gamma}{T-t}\,W(t),
    \label{eq:final_dissipation}
\end{equation}
where $\gamma>0$ can be taken arbitrarily close to the centralized rate $2\underline{\sigma}^2\mu_c$ as $c_o\to\infty$.

\subsubsection*{Step 5: Trajectory Confinement and Conclusion}

The dissipation inequality~\eqref{eq:final_dissipation} was derived under the standing assumption that $W(t) \leq W(0)$, which ensures that all constants are evaluated on the compact set $\Omega_{W(0)}$. We now verify this assumption. Suppose for contradiction that $W(t_0) > W(0)$ for some $t_0 \in (0,T)$, and define $t^* = \inf\{t > 0 : W(t) > W(0)\}$. By continuity, $W(t) \leq W(0)$ on $[0, t^*]$ with $W(t^*) = W(0)$, so~\eqref{eq:final_dissipation} holds on this interval. Integrating yields $W(t^*) \leq W(0)((T - t^*)/T)^\gamma < W(0)$, contradicting $W(t^*) = W(0)$. Therefore $W(t) \leq W(0)$ for all $t \in [0, T)$, and~\eqref{eq:final_dissipation} holds throughout.

Solving the differential inequality~\eqref{eq:final_dissipation} by direct integration:
\begin{equation}
    W(t) \leq W(0) \left(\frac{T - t}{T}\right)^\gamma.
    \label{eq:W_bound}
\end{equation}

Since $(T-t)/T \to 0$ as $t \to T^-$, we have $\lim_{t \to T^-} W(t) = 0$, which implies $W_c(T) = 0$, $W_o(T) = 0$, and $W_\delta(T) = 0$. The first gives $\be_{ij}(T) = 0$ for all $i \neq j$ (consensus); the second gives $V(\bx, \bar{\blambda}, \bar{\bmu}) = 0$ and hence $\bS(\bx, \bar{\blambda}, \bar{\bmu}) = 0$ (optimality); the third gives $\blambda_i(T) = \blambda_j(T)$ and $\bmu_i(T) = \bmu_j(T)$ for all $i, j$ (v-GNE).

Since $W(t) \to 0$ monotonically and $\Omega_{W(0)}$ is compact, $\bS(\bx, \bar{\blambda}, \bar{\bmu}) \to 0$ with $\bar{\blambda} = \blambda_i$ and $\bar{\bmu} = \bmu_i$ for all $i$ at $t = T$. By uniqueness of the v-GNE under strong monotonicity~\cite{facchinei_generalized_2010}, the primal solution $\bx^*$ is unique. The full row rank of $A$ and the strict positivity of $\Phi_\epsilon$ then uniquely determine the dual variables $(\blambda^*, \bmu^*)$, so $\bS$ has a unique zero. This yields $(\bx(T), \blambda_i(T), \bmu_i(T)) = \bz^*$ for all $i$, confirming convergence to the unique v-GNE.

Boundedness on $[0, T)$ follows from~\eqref{eq:W_bound}: since $W(t) \leq W(0)$ for all $t \in [0,T)$, the centralized state $(\bx, \bar{\blambda}, \bar{\bmu})$ remains in the compact sublevel set $\Omega_{W(0)}$, and the dual disagreement satisfies $\|\bdelta\|^2 \leq 2W_\delta/\lambda_2 \leq 2W(0)/(k_d \lambda_2)$, so every agent's trajectory remains bounded. \hfill $\blacksquare$

The theoretical gains~\eqref{eq:consensus_scaling}--\eqref{eq:dual_consensus_scaling} diverge at $t = T$, which cannot be realized in digital implementation. In practice, one replaces $(T-t)^{-1}$ with $(T - t + \bar{\varepsilon})^{-1}$ for a small $\bar{\varepsilon} > 0$. Unlike the Fischer-Burmeister smoothing~$\epsilon$, which is an algorithmic design choice determining the equilibrium, $\bar{\varepsilon}$ is purely a numerical artifact. The regularized analysis yields $W(T) = O(\bar{\varepsilon}^\gamma)$.

\section{Numerical Simulations}
\label{sec:simulations}

We validate the proposed prescribed-time distributed GNE (PT-DGNE) algorithm in two scenarios: a Cournot competition game~\cite{martinez-piazuelo_payoff_2022} and a time-critical sensor coverage application.

\subsection{Cournot Competition Game}
We validate the algorithm on a Networked Cournot Competition with both equality and inequality constraints, using the minimal connectivity of a spanning tree topology.

\subsubsection{Setup and Parameters}
Consider $N=20$ firms, each choosing production level $x_i \in \mathbb{R}$ to maximize profit:
\begin{equation}
    J_i(x_i, \mathbf{x}_{-i}) = \underbrace{(\alpha_i x_i^2 + \beta_i x_i)}_{\text{Production Cost}} - \underbrace{x_i \left(P_0 - d \sum_{j=1}^N x_j\right)}_{\text{Revenue}},
\end{equation}
where $P_0 = 50$ is the base price, $d = 0.2$ is the price elasticity, and the cost coefficients are randomly drawn from $\alpha_i \in [1, 2]$ and $\beta_i \in [5, 10]$.

The agents are subject to shared coupling constraints:
\begin{align}
    &\sum_{i=1}^N x_i \leq C_{max}, \qquad \text{(Infrastructure Capacity)} \\
    &\sum_{i=1}^N r_i x_i = R_{target}, \qquad \text{(Regulatory Quota)}
\end{align}
where $C_{max} = 40$ and $R_{target} = 40$. The weights $r_i \in [0.8, 1.2]$ represent agent-specific efficiency factors. With these parameters, the inequality (capacity) constraint is active at the equilibrium, while the equality constraint is active by definition. The prescribed convergence time is $T = 10$\,s with gains $\mu_c = 20$, $k_o = 50$, $c_o = 100$, $k_d = 5$, and consensus exponent $\gamma_c = 2$. The Fischer-Burmeister smoothing parameter is $\epsilon = 10^{-8}$, regularized with $\bar{\varepsilon} = 10^{-10}$, and the dynamics are integrated with \textit{ode15s} (\texttt{RelTol} $10^{-8}$, \texttt{AbsTol} $10^{-13}$, \texttt{MaxStep} $0.002$) under a fixed random seed. Since the inequality constraints are affine, $c^* = +\infty$ and the initial condition of Thm.~\ref{thm:main} imposes no restriction.

\subsubsection{Results}

To verify the algorithm's robustness to sparse connectivity, the agents communicate over a spanning tree topology shown in Fig.~\ref{fig:topology}. Despite the low algebraic connectivity ($\lambda_2 \approx 0.096$), the prescribed-time consensus protocol ensures global information is shared effectively before the deadline.

\begin{figure}[ht]
    \centering
    \includegraphics[width=0.65\columnwidth]{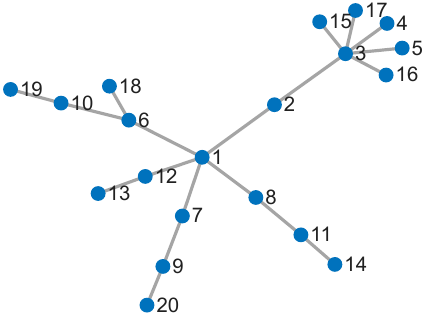}
    \caption{Communication topology: A tree graph with $N=20$ agents. This minimal connectivity represents a challenging scenario for distributed consensus.}
    \label{fig:topology}
\end{figure}

Figure~\ref{fig:primal_traj} displays the trajectories of the decision variables $x_i(t)$ for all 20 agents. Starting from identical initial conditions $x_i(0) = 5$, chosen to confirm that production differentiation arises from the optimization dynamics alone, the agents rapidly adjust their production levels. As $t \to T$, the time-varying gains enforce exact convergence to the equilibrium values $x_i^* \in [1.38, 3.07]$. The trajectories are smooth and bounded throughout the interval $[0, T)$, and the final primal error is $\|\mathbf{x}(T) - \mathbf{x}^*\| \approx 10^{-11}$.

\begin{figure}[ht]
    \centering
    \includegraphics[width=.99\columnwidth]{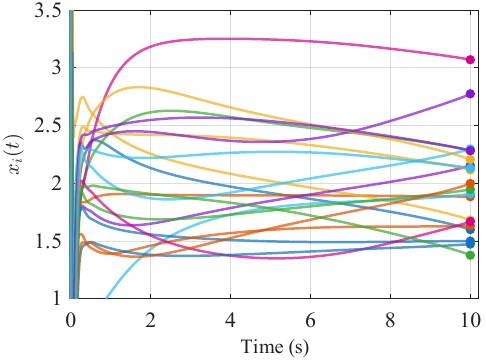}
    \caption{Trajectories of decision variables $x_i(t)$ for $N=20$ agents. All agents converge to their optimal production levels exactly at $T=10$\,s.}
    \label{fig:primal_traj}
\end{figure}

Figure~\ref{fig:stationarity} shows the stationarity residual $\|\bS(\mathbf{z}(t))\|$, capturing the combined error in primal optimality, dual feasibility, and complementarity. Both the centralized benchmark and the distributed algorithm exhibit the characteristic waterfall shape of prescribed-time methods, with the residual dropping to $\approx 10^{-10}$ at the deadline. The constraint residuals at $T$ are both below $10^{-10}$, with the capacity constraint active at the solution.

\begin{figure}[ht]
    \centering
    \includegraphics[width=.99\columnwidth]{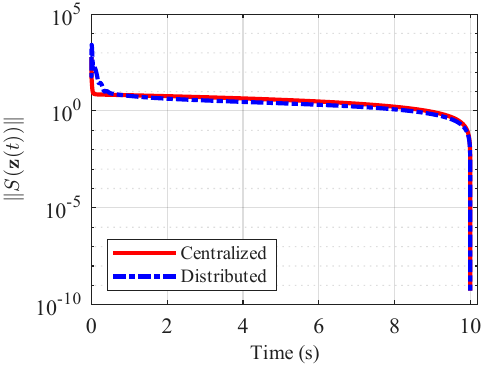}
    \caption{Stationarity residual $\|\bS(\mathbf{z}(t))\|$ over time. The residual decays to zero, indicating that the KKT conditions for the v-GNE are satisfied at $T$.}
    \label{fig:stationarity}
\end{figure}

Convergence to a v-GNE additionally requires consensus on the dual variables: starting from randomized initial values $\lambda_i(0) \sim \mathcal{U}[0.5,\, 1.5]$ and $\mu_i(0) \sim \mathcal{U}[10,\, 15]$, the maximum disagreement between agents' local multiplier estimates decays to $\approx 10^{-10}$, confirming that the dual consensus dynamics dominate the optimization-induced drift and that all agents converge to the shared shadow prices $\lambda^* \approx 32.92$ and $\mu^* \approx -4.36$.

\subsection{Time-Critical Sensor Coverage}
\label{subsec:sensor_network}

We now apply the algorithm to an engineering problem: optimal deployment of a mobile sensor network under shared resource constraints. Unlike the economic example, this scenario represents a solver-in-the-loop setup where the optimization must complete exactly within the sampling interval $T_{sample} = 1$~s to generate valid waypoints for vehicle autopilots.

\subsubsection{Setup, Parameters, and Results}
Consider $N=20$ mobile sensors operating in a 2D plane, $\mathbb{R}^2$. The agents communicate over a spanning tree topology shown in Fig.~\ref{fig:topology}. Each sensor $i$ aims to minimize a cost function balancing target tracking and network connectivity:
\begin{equation}
    J_i(\mathbf{x}_i, \mathbf{x}_{-i}) = \frac{1}{2} \|\mathbf{x}_i - \mathbf{b}_i\|^2 + \sum_{j \in \mathcal{N}_i} \frac{a_{ij}}{2} \|\mathbf{x}_i - \mathbf{x}_j\|^2,
\end{equation}
where $\mathbf{x}_i \in \mathbb{R}^2$ is the position of sensor $i$, $\mathbf{b}_i$ is the location of the target assigned to sensor $i$, and $a_{ij}$ are the entries of the adjacency matrix.

The system is coupled by a shared power budget constraint, e.g., a total battery capacity or tethering limit. This introduces a global inequality constraint:
\begin{equation}
    g(\mathbf{x}) = \sum_{i=1}^N \|\mathbf{x}_i\|^2 - P_{total} \le 0.
\end{equation}
We position the targets $\mathbf{b}_i$ on a circle of radius $R=15$m, while setting $P_{total}$ such that the maximum average deployment radius is $R_{max}=10$m. This configuration forces the unconstrained optimum to be infeasible, necessitating a GNE solution on the boundary of the feasible set.

The sensor coverage problem has a nonlinear quadratic constraint, and its shared feasible set $\{\bx : \sum_i \|\bx_i\|^2 \leq P_{\text{total}}\}$ is a ball, hence bounded, so a finite threshold $c^* > 0$ exists by Proposition~\ref{prop:sufficient_radial}. Because the constraint is nonlinear, the guarantee here is semi-global, requiring $W(0) < c^*$; as Proposition~\ref{prop:sufficient_radial} establishes $c^*$ only as an existence result, it is not computed explicitly, and the convergence reported below serves as an a posteriori certificate that this condition holds. As in the Cournot example, the consensus and dual gains $c_o = 100$, $k_d = 5$ and consensus exponent $\gamma_c = 2$ are used, with optimization gains $\mu_c = 13.63$, $k_o = 50$, and Fischer-Burmeister smoothing $\epsilon = 10^{-8}$; the simulation results below confirm that these gains satisfy the gain-dominance requirement of Theorem~\ref{thm:main} for this structurally different problem. Since $\nabla^2 g = 2I_n$ and the smoothed multiplier $\widetilde{\lambda}$ is positive, the curvature term $2\widetilde{\lambda} I_n \succeq 0$ is automatic, so the generalized Lagrangian Hessian is positive definite by strong monotonicity alone, with no condition to verify numerically.

We note that the pairwise coupling terms $\|\mathbf{x}_i - \mathbf{x}_j\|^2$ in the cost function prevent this problem from being formulated as an aggregative game, since each agent requires individual neighbor positions rather than a scalar aggregate. Consequently, algorithms designed for aggregative games, such as~\cite{guo_distributed_2025}, are not applicable to this scenario.

The algorithm is executed with a prescribed time $T=1$~s. Sensors are initialized at the origin with dual estimates $\lambda_i(0) \sim \mathcal{U}[0.05, 0.15]$, the power budget is $P_{\text{total}} = N R_{max}^2 = 2000$, and the regularization parameter is $\bar{\varepsilon} = 10^{-10}$. The integration uses MATLAB's built-in \textit{ode15s} with an analytical Jacobian (\texttt{RelTol} $10^{-9}$, \texttt{AbsTol} $10^{-11}$, \texttt{MaxStep} $0.1$). The results are compared against a centralized prescribed-time solver implementing~\eqref{eq:centralized_GD} with the same $T$. Wall-clock computation times on a standard desktop are $1.15$~s for the centralized solver and $0.66$~s per agent for the distributed algorithm. In a parallel implementation with dedicated hardware per agent, the distributed wall-clock time is well within the $T=1$~s sampling window, whereas the centralized solver exceeds it.

Figure~\ref{fig:sensor_convergence} compares the norm of the stationarity residual $\|\bS(\mathbf{z})\|$. Both the centralized solver and the proposed PT-DGNE algorithm drive the residual below $10^{-8}$ at $t=T$. The vertical drop on the log-scale plot confirms that the optimization process effectively terminates at the prescribed deadline, providing a mathematically exact solution for the control loop.

\begin{figure}[ht]
    \centering
    \includegraphics[width=.99\columnwidth]{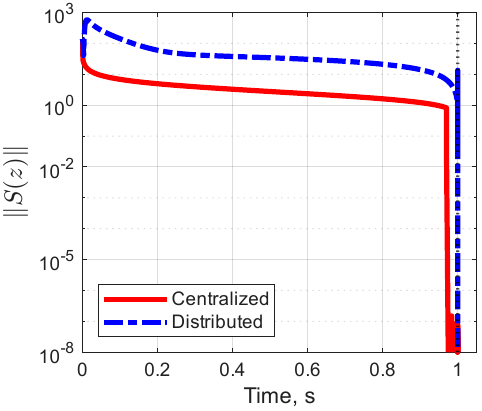}
    \caption{Convergence of the Stationarity vector for the Sensor Network. Note the vertical drop below $\approx 10^{-8}$ at exactly $T=1$~s, validating the prescribed-time property.}
    \label{fig:sensor_convergence}
\end{figure}

Figure~\ref{fig:sensor_deployment} illustrates the final configuration at $t=T$. The tension is visually evident: sensors (blue) attempt to reach their targets (red) but are constrained by the power budget (dashed line). The agents automatically distribute themselves along the active constraint boundary. The shared constraint $g(\mathbf{x})$ is maintained throughout $[0,T)$ and becomes active exactly at the deadline without violation, as certified by the $S_3$ block of the stationarity residual converging to zero in Fig.~\ref{fig:sensor_convergence}. Convergence to a v-GNE further requires the per-agent multipliers to agree; the dual-consensus layer drives the maximum disagreement $\max_{i,j}\|\lambda_i - \lambda_j\|$ to machine precision by the deadline, so all sensors recover the common shadow price $\lambda^* \approx 0.18$ for the power budget, in tandem with the stationarity convergence of Fig.~\ref{fig:sensor_convergence}.

\begin{figure}[ht]
    \centering
    \includegraphics[width=0.99\columnwidth]{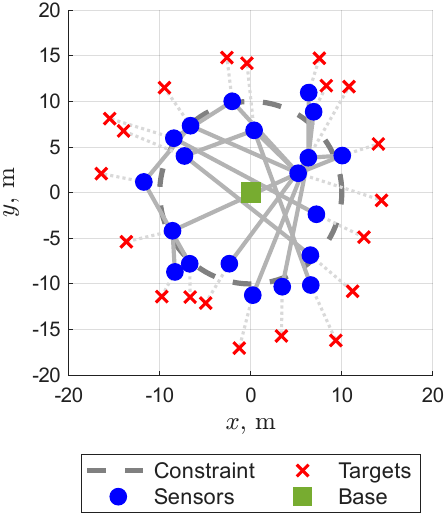}
    \caption{Sensor Deployment at $t=T$. The sensors are distributed along the boundary of the active power constraint, balancing target tracking (dotted lines).}
    \label{fig:sensor_deployment}
\end{figure}



\section{Conclusions}
\label{sec:conclusion}

This paper addressed GNE seeking under a hard real-time constraint, a regime in which asymptotic and fixed-time guarantees fall short because the solution is useful only if it arrives by a known deadline. By casting the variational GNE conditions as the zero of a stationarity map and driving that map to zero with a time-varying feedback, we obtained convergence at a user-chosen time $T$ while handling the shared coupling constraints and general cost coupling that the existing prescribed-time NE literature does not.

A central lesson of the distributed design is that prescribed-time optimization and prescribed-time consensus cannot simply be run in parallel and summed. Because each agent optimizes against an estimate of a global state that the optimization itself keeps moving, the estimation, optimization, and dual-agreement layers perturb one another throughout the transient. A multi-rate gain schedule, in which the observer and dual-consensus layers diverge faster than the optimization layer, is what decouples these effects asymptotically and lets a single composite Lyapunov function certify that all error components vanish together at $T$. The resulting gain-dominance conditions are mild, since they can be met on any connected graph by taking the consensus and dual gains large enough.

The sensor-coverage study shows the practical payoff most directly. Acting as a solver in the loop with one processor per agent, the distributed algorithm returns a feasible, constraint-satisfying waypoint set within the sampling interval, whereas the centralized solver exceeds it, which is exactly the setting prescribed-time guarantees are meant for. The price of the guarantee is communication: each agent estimates the full network state, so the per-step load grows with the network size, and the time-varying gains must be regularized near $T$ for digital implementation. Relaxing the full-state estimation requirement, extending the architecture to directed or time-varying graphs, and adding event-triggered communication to reduce overhead are natural next steps.


\appendices

\section{Proof of Lemma~\ref{lem:nondeg}}
\label{app:nondeg}
\begin{proof}
The Jacobian of $\mathbf{S}(\bz)$ with respect to $\bz = \mathrm{col}(\bx,\blambda,\bmu)$ has the block structure
\begin{equation}
    \nabla\mathbf{S}(\bz) = \begin{pmatrix} H_L(\bx,\blambda) & \nabla\mathbf{g}(\bx)^\top D_{\widetilde{\lambda}'} & A^\top \\ A & 0 & 0 \\ -D_g\nabla\mathbf{g}(\bx) & D_\lambda & 0 \end{pmatrix},
    \label{eq:nablaS_block}
\end{equation}
where 
\begin{equation}
    H_L(\bx,\blambda) = \nabla_{\bx}\mathbf{F}(\bx) + \sum_{j=1}^{p} \widetilde{\lambda}_j \nabla^2 g_j(\bx)    
\end{equation}
is the generalized Lagrangian operator, $D_{\widetilde{\lambda}'} = \mathrm{diag}(\widetilde{\lambda}_j')$ with $\widetilde{\lambda}_j' \in (0,1)$ arising from differentiating $\widetilde{\lambda}_j$ in~\eqref{eq:lambda_tilde} with respect to $\lambda_j$, $D_g = \mathrm{diag}(d^g_j)$ with 
\begin{equation}
    d^g_j = -g_j(\bx)/\sqrt{\lambda_j^2 + g_j(\bx)^2 + \epsilon^2} - 1
\end{equation} 
collects the derivative of $\Phi_\epsilon$ with respect to its second argument (the $(3,1)$ block $-D_g\nabla\mathbf{g}(\bx)$), and $D_\lambda = \mathrm{diag}(d^\lambda_j)$ with 
\begin{equation}
    d^\lambda_j = \lambda_j/\sqrt{\lambda_j^2 + g_j(\bx)^2 + \epsilon^2} - 1
\end{equation} 
collects the derivative with respect to its first argument (the $(3,2)$ block $D_\lambda$). For $\epsilon > 0$, both lie strictly in $(-2,0)$ at every point, so $D_\lambda \prec 0$ is invertible.

Since $\widetilde{\lambda}_j > 0$ and each $g_j$ is convex, $\sum_j \widetilde{\lambda}_j \nabla^2 g_j \succeq 0$, so strong monotonicity (Assumption~\ref{ass:monotonicity}) gives $v_1^\top H_L(\bx,\blambda)\,v_1 \geq m_F\|v_1\|^2$ for all $v_1$, with no sign assumption on $\blambda$. With $D_{\widetilde{\lambda}'} \succ 0$, $D_\lambda \prec 0$, and $D_g \prec 0$, a block-elimination argument on this smoothed Jacobian shows $\nabla\mathbf{S}(\bz)\,v = 0 \Rightarrow v = 0$, hence $\nabla\mathbf{S}(\bz)$ is non-singular at every $\bz$; this is the GNE Jacobian regularity result of~\cite{mudrik_optimization_2025-1}.

The function $\bz \mapsto \sigma_{\min}(\nabla\mathbf{S}(\bz))$ is continuous (as $\mathbf{S}$ is twice continuously differentiable by Assumption~\ref{ass:smoothness}) and strictly positive on $\Omega_c$. Since $\Omega_c$ is compact (Proposition~\ref{prop:sufficient_radial}), the infimum is attained and $\underline{\sigma} > 0$.
\end{proof}

\section{Proof of Proposition~\ref{prop:centralized_PT}}
\label{app:centralized_PT}
Since $V(\bz(0)) < c^*$, the sublevel set $\Omega_{V(0)}$ is compact by Proposition~\ref{prop:sufficient_radial}, and Lemma~\ref{lem:nondeg} yields the uniform lower bound $\underline{\sigma} > 0$ on this set. Using $\nabla V = \nabla\mathbf{S}(\bz)^\top\mathbf{S}(\bz)$ and the minimum singular value bound gives the PL condition $\|\nabla V(\bz)\|^2 \geq \underline{\sigma}^2\|\mathbf{S}(\bz)\|^2 = 2\underline{\sigma}^2 V(\bz)$, so along~\eqref{eq:centralized_GD}:
\begin{equation}
    \dot{V} = -\sigma_{opt}(t)\|\nabla V\|^2 \leq -\frac{2\underline{\sigma}^2\mu_c}{T-t}\,V(\bz).
    \label{eq:centralized_dissipation}
\end{equation}
A contradiction argument identical to Step~5 of the proof of Theorem~\ref{thm:main} (with $W_c = W_\delta = 0$ and $W = W_o = V$) establishes $V(\bz(t)) \leq V(\bz(0))$ for all $t \in [0,T)$, after which integrating~\eqref{eq:centralized_dissipation} gives~\eqref{eq:centralized_decay}. 
We obtain finite arc length 
\begin{equation}
    \int_0^T\|\dot{\bz}\|\,dt \leq C\int_0^T(T{-}t)^{\gamma/2-1}dt < \infty
\end{equation}
since $\gamma > 0$ confirms that $\bz(t)$ converges to the unique zero $\bz^*$ of $\mathbf{S}$. \hfill$\blacksquare$

\bibliographystyle{IEEEtran}
\bibliography{Bib1}

\end{document}